\documentclass[11pt,twoside,a4paper]{article}
\usepackage{color,amscd,amsmath,amssymb,amsthm,latexsym,stmaryrd,authblk,mathabx,shuffle,mathrsfs,hyperref,tikz,tkz-tab} 
\usepackage[latin1]{inputenc}
\usepackage[T1]{fontenc}   
\usepackage[english]{babel}
\providecommand{\AMSclass}[1]{\textbf{\textit{AMS classification.}} #1}

\input{xy}
\xyoption{all}

\title{Bialgebras overs another bialgebras and quasishuffle double bialgebras}
\date{}
\author{Lo\"ic Foissy}
\affil{\small{Univ. Littoral Côte d'Opale, UR 2597
LMPA, Laboratoire de Mathématiques Pures et Appliquées Joseph Liouville
F-62100 Calais, France}.\\ Email: \texttt{foissy@univ-littoral.fr}}

\newcommand{\tdun}[1]{\begin{picture}(10,5)(-2,-1)
\put(0,0){\circle*{2}}
\put(3,-2){\tiny #1}
\end{picture}}

\newcommand{\tddeux}[2]{\begin{picture}(12,5)(0,-1)
\put(3,0){\circle*{2}}
\put(3,5){\circle*{2}}
\put(3,0){\line(0,1){5}}
\put(6,-3){\tiny #1}
\put(6,5){\tiny #2}
\end{picture}}

\newcommand{\gdtroisun}[3]{\begin{picture}(22,12)(-8,-1)
\put(3,0){\circle*{2}}
\put(6.5,7){\circle*{2}}
\put(-1,7){\circle*{2}}
\put(-1,7){\line(1,0){7.5}}
\put(-2.8,0){\Large $\vee$}
\put(5,-2){\tiny #1}
\put(9,5){\tiny #2}
\put(-8,5){\tiny #3}
\end{picture}}

\newcommand{\tdtroisdeux}[3]{\begin{picture}(12,15)(-2,-1)
\put(0,0){\circle*{2}}
\put(0,5){\circle*{2}}
\put(0,10){\circle*{2}}
\put(0,0){\line(0,1){5}}
\put(0,5){\line(0,1){5}}
\put(3,-4){\tiny #1}
\put(3,4){\tiny #2}
\put(3,12){\tiny #3}
\end{picture}}

\theoremstyle{plain}
\newtheorem{theo}{Theorem}[section]
\newtheorem{lemma}[theo]{Lemma}
\newtheorem{cor}[theo]{Corollary}
\newtheorem{prop}[theo]{Proposition}
\newtheorem{defi}[theo]{Definition}

\theoremstyle{remark}
\newtheorem{remark}{Remark}[section]
\newtheorem{notation}{Notations}[section]
\newtheorem{example}{Example}[section]

\newcommand{\K}{\mathbb{K}}
\newcommand{\N}{\mathbb{N}}
\newcommand{\Z}{\mathbb{Z}}

\newcommand{\bfG}{\mathbf{G}}

\newcommand{\id}{\mathrm{Id}}

\newcommand{\eq}{\mathcal{E}}

\newcommand{\Char}{\mathrm{Char}}

\newcommand{\tdelta}{\tilde{\Delta}}

\newcommand{\vect}{\mathrm{Vect}}
\renewcommand{\ker}{\mathrm{Ker}}

\newcommand{\QSh}{\mathrm{QSh}}
\newcommand{\Sh}{\mathrm{Sh}}
\newcommand{\QSym}{\mathbf{QSym}}
\newcommand{\Hom}{\mathrm{Hom}}

\newcommand{\calH}{\mathcal{H}}

\newcommand{\col}{\mathcal{C}}
\newcommand{\qsh}{\mathrm{QSh}}
\newcommand{\morH}{\Upsilon}

\begin{document}

\maketitle

\begin{abstract}
Quasishuffle Hopf algebras, usually defined on a commutative monoid, can be more generally defined on any associative algebra $V$.
If $V$ is a commutative and cocommutative bialgebra, the associated quasishuffle bialgebra $\qsh(V)$ inherits a second coproduct
$\delta$ of contraction and extraction of words, cointeracting with the deconcatenation coproduct $\Delta$, 
making $\qsh(V)$ a double bialgebra.
In order to generalize the universal property of the Hopf algebra of quasisymmetric functions $\QSym$ (a particular case
of quasishuffle Hopf algebra) as exposed by Aguiar, Bergeron and Sottile, we introduce the notion of double bialgebra over $V$.
A bialgebra over $V$ is a bialgebra in the category of right $V$-comodules and an extra condition is required on the second coproduct
for double bialgebras over $V$.

We prove that the quasishuffle bialgebra $\qsh(V)$ is a double bialgebra over $V$, and that it satisfies a universal property:
for any bialgebra $B$ over $V$ and for any character $\lambda$ of $B$, under a connectedness condition, there exists
a unique morphism $\phi$ of bialgebras over $V$ from $B$ to $\qsh(V)$ such that $\varepsilon_\delta \circ \phi=\lambda$. 
When $V$ is a double bialgebra over $V$, we obtain a unique morphism of double bialgebras over $V$ from $B$
to $\qsh(V)$, and show that this morphism $\phi_1$ allows to obtain any morphism of bialgebra over $V$ from $B$
to $\qsh(V)$ thanks to an action of a monoid of characters. This formalism is applied to a double bialgebra of $V$-decorated graphs. 
\end{abstract}

\AMSclass{16T05 05A05 68R15 16T30}

\tableofcontents

\section*{Introduction}

Quasishuffle bialgebras are Hopf algebras based on words, used in particular for the study of relations between multizêtas
\cite{Hoffman2000,Hoffman2020}. They also appear in Ecalle's mould calculus,
as a symmetrel mould can be interpreted as a character on a quasishuffle bialgebras  \cite{Ebrahimi-Fard2017-2}. 
Hoffman's construction is based on commutative countable semigroups, but it can be extended to any associative algebra $(V,\cdot)$,
not necessarily unitary \cite{Foissy40}. The associated quasishuffle bialgebra $\qsh(V)$ is, as a vector space, the tensor algebra $T(V)$. 
Its product is the quasishuffle product  $\squplus$, inductively defined as follows: if $x,y\in V$ and $v,w\in T(V)$,
\begin{align*}
1\squplus w&=w,\\
v\squplus 1&=v,\\
xv\squplus yw&=x(v\squplus yw)+y(xv\squplus w)+(x\cdot y)(v\squplus w). 
\end{align*}
For example, if $x,y,z,t\in V$,
\begin{align*}
x\squplus y&=xy+yx+x\cdot y,\\
xy\squplus z&=xyz+xzy+zxy+(x\cdot z)y+x(y\cdot z),\\
xy\squplus zt&=xyzt+xzyt+zxyt+xzty+zxty+ztxy\\
&+(x\cdot z)ty+(x\cdot z)yt+xz(y\cdot t)+zx(y\cdot t)+(x\cdot z)(y\cdot t). 
\end{align*}
The coproduct $\Delta$ is the deconcatenation: if $x_1,\ldots,x_n \in V$,
\[\Delta(x_1\ldots x_n)=\sum_{i=0}^n x_1\ldots x_i\otimes x_{i+1}\ldots x_n.\]
When $(V,\cdot,\delta_V)$ is a commutative bialgebra, not necessarily unitary, then $\qsh(V)$ inherits a second, 
less known coproduct $\delta$: if $x_1,\ldots,x_n \in V$,
\[\delta(v_1 \ldots  v_n)=\sum_{1\leq i_1<\ldots<i_p<k}
\left(\prod_{1\leq i\leq i_1}^\cdot v'_i\right)\ldots  \left(\prod_{i_p+1\leq i\leq k}^\cdot v'_i\right)
\otimes (v''_1 \ldots v''_{i_1})\squplus \ldots \squplus (v''_{i_p+1} \ldots  v''_k),\]
with Sweedler's notation for $\delta_V$ and where the symbols $\displaystyle \prod^\cdot$ 
mean that the products are taken in $(V,\cdot)$.
The counit $\epsilon_\delta$ is given as follows: for any word $w$ of length $n\geq 1$,
\[\epsilon_\delta(w)=\begin{cases}
\epsilon_V(w)\mbox{ if }n=1,\\
0\mbox{ otherwise}. 
\end{cases}\]
Then $(T(V),\squplus,\delta)$ is a bialgebra, and $(T(V),\squplus,\Delta)$ is a bialgebra in the category of right
$(T(V),\squplus,\delta)$-comodules, which in particular implies that 
\[(\Delta \otimes \id)\circ \delta=\squplus_{1,3,24}\circ (\delta \otimes \delta)\circ \Delta,\]
where $\squplus_{1,3,24}:T(V)^{\otimes 4}\longrightarrow T(V)^{\otimes 3}$ send $w_1\otimes w_2\otimes w_3 \otimes w_4$
to $w_1\otimes w_3\otimes w_2\squplus w_4$. 
Two particular cases will be considered all along this paper:
\begin{itemize}
\item $V=\K$, with its usual bialgebraic structure. The quasishuffle algebra $\qsh(\K)$ is isomorphic
to the polynomial algebra $\K[X]$, with its two coproducts defined by
\begin{align*}
\Delta(X)&=X\otimes 1+1\otimes X,&\delta(X)&=X\otimes X.
\end{align*}
\item $V$ is the algebra of the semigroup $(\N_{>0},+)$. We recover the double Hopf algebra
of quasisymmetric functions $\QSym$ \cite{Gelfand1995,Hazewinkel2005,Malvenuto2011,Stanley1999}.
This Hopf algebra is studied in \cite{Aguiar2006-2}, where it is proved to be the terminal object in a category of combinatorial 
Hopf algebras: If $B$ is a graded and connected Hopf algebra and $\lambda$ is a character of $B$,
then there exists a unique homogeneous Hopf algebra morphism $\phi_\lambda:B\longrightarrow \QSym$
such that $\epsilon_\delta \circ \phi_\lambda=\lambda$. We proved in \cite{Foissy37,Foissy36} that when $(B,m,\Delta,\delta)$
is a double bialgebra, such that:
\begin{itemize}
\item $(B,m,\Delta)$ is a graded and connected Hopf algebra,
\item for any $n\in \N$, $\delta(B_n)\subseteq B_n\otimes B$, 
\end{itemize} 
then $\phi_{\varepsilon_\delta}$ is the unique homogeneous double bialgebra
morphism from $B$ to $\QSym$. A similar result exists for $\K[X]$, where the hypothesis "graded and connected" 
on $B$ is replaced by the weaker hypothesis "connected".
\end{itemize}

In this paper, we generalize these results to any quasishuffle $\qsh(V)$ associated to a commutative and cocommutative
bialgebra $(V,\cdot,\delta_V)$, not necessarily unitary. We firstly show that $(T(V),\cdot,\Delta)$ 
is a bialgebra in the category of right $V$-comodules, with the coaction $\rho$ defined by
\begin{align*}
&\forall v_1,\ldots,v_n \in V,&
\rho(v_1\ldots v_n)&=v_1'\ldots v_n'\otimes v_1''\cdot \ldots \cdot v_n''.
\end{align*}
Moreover, the second coproduct $\delta$ satisfies this compatibility with $\rho$:
\[(\id \otimes c)\circ (\rho\otimes \id)\circ \delta=(\delta \otimes \id)\circ \rho,\]
where $c:V\otimes T(V)\longrightarrow T(V)\otimes V$ is the usual flip. Equivalently, $(T(V),\squplus,\Delta)$ is a comodule
over the coalgebra $(V,\delta_V^{op})\otimes (T(V),\delta)$. This observation leads us to study bialgebras over $V$, that is to say
bialgebras in the category of right $(V,\cdot,\delta_V)$-comodules (Definition \ref{defioverV} when $V$ is unitary).
Technical difficulties occur when $V$ is not unitary, a  case that cannot be neglected as it includes $\QSym$:
this is the object of Definition \ref{Defi1.7}, where we use the unitary extension $uV$ of $V$, which is also a bialgebra.
We define double bialgebras over $V$ in Definition \ref{defi1.4} in the unitary case and Definition \ref{Defi1.7} in the nonunitary case.
When $V=\K$, bialgebras over $V$ are bialgebras $B $ with a decomposition $B=B_1\oplus B_{\overline{1}}$,
where $B_1$ is a subbialgebra and $B_{\overline{1}}$ is a biideal. This includes any bialgebra $B$, taking $B_1=\K1_B$ and $B_{\overline{1}}$ the kernel
of the counit.
 When $V=\K(\N_{>0},+)$, bialgebras over $V$ are $\N$-graded and connected bialgebras,
in other words $\N$-graded bialgebras $B$ with $B_0=\K1_B$. 

We prove that the antipode of a bialgebra $(B,m,\Delta,\rho)$ over $V$, such that $(B,m,\Delta)$ is a Hopf algebra,
is automatically a comodule morphism (Proposition \ref{prop1.2}), that is to say
\[\rho\circ S=(S\otimes \id_V)\circ \rho.\]
In the case of $\N$-graded bialgebras, this means that $S$ is automatically homogeneous; more generally,
if $\Omega$ is a commutative semigroup and $B$ is an $\Omega$-graded bialgebra and a Hopf algebra, then its antipode
is automatically $\Omega$-homogeneous. 

Let us now consider the double quasishuffle algebra $\qsh(V)=(T(V),\squplus,m,\Delta,\delta)$, 
which is over $V$ with the coaction $\rho$.
We obtain a generalization of Aguiar, Bergeron and Sottile's result: Theorem \ref{theomorphismes} states that for any 
connected bialgebra $B$ over $V$ and for any character $\lambda$ of $B$, there exists a unique morphism $\phi_\lambda$
from $B$ to $\qsh(V)$ of bialgebras over $V$ such that $\epsilon_\delta \circ \phi_\lambda=\lambda$, given by an explicit formula
implying the iterations of the reduced coproduct $\tdelta$ associated to the coproduct $\Delta$ of $B$. 

When $B$ is moreover a double bialgebra over $V$, we prove that the unique morphism of double bialgebras over $V$
from $B$ to $\qsh(V)$ is $\Phi_{\epsilon_\delta}$ (Theorem \ref{theo7.10}).
Moreover, for any bialgebra $B'$ over $V$, the second coproduct $\delta$
induces an action $\leftsquigarrow$ of the monoid of characters $\Char(B)$ (with the product induced by $\delta$)
onto the set of morphisms of bialgebras over $V$ from $B$ to $B'$ (Proposition \ref{prop5.19}. When $B'=\qsh(V)$,
we obtain that this action is simply transitive (Corollary \ref{cormorphismedouble}), which gives a bijection
between the set of characters of $B$ and the set of morphisms of double bialgebras over $V$ from $B$ to $\qsh(V)$. 
This is finally applied to the twisted bialgebra of graphs $\bfG$: for any $V$, we obtain a double bialgebra 
$\calH_V$ of $V$-decorated graphs, and the unique morphism of double bialgebras over $V$ from
$\calH_V$ to $\qsh(V)$ is a generalization of the chromatic polynomial and of the chromatic (quasi)symmetric series. 
Taking $V=\K$ or $\K(\N_{>0},+)$, we recover the terminal property ok $\K[X]$ and $\QSym$. \\

\textbf{Acknowledgements}. 
%The author thanks the anonymous referee for his or her helpful comments. 
The author acknowledges support from the grant ANR-20-CE40-0007
\emph{Combinatoire Algébrique, Résurgence, Probabilités Libres et Opérades}.\\

\begin{notation} \begin{enumerate}
\item We denote by $\K$ a commutative field of characteristic zero. Any vector space in this field will be taken over $\K$.
\item For any $n\in \N$, we denote by $[n]$ the set $\{1,\ldots,n\}$. In particular, $[0]=\emptyset$.
\end{enumerate}\end{notation}

\section{Bialgebras over another bialgebra}

\subsection{Définitions and notations}

Let $(V,\cdot,\delta_V)$ be a commutative bialgebra, which we firstly assume to be unitary and counitary.
Its counit is denoted by $\epsilon_V$ and its unit by $1_V$.

\begin{defi} \label{defioverV}
A bialgebra over $V$ is a bialgebra in the category of right $V$-comodules, that is to say a family $(B,m,\Delta,\rho)$ 
where $(B,m,\Delta)$ is a bialgebra and $\rho:B\longrightarrow B\otimes V$ such that:
\begin{itemize}
\item $\rho$ is a right coaction of $V$ over $B$:
\begin{align*}
(\rho\otimes \id_V)\circ \rho&=(\id_B\otimes \delta_V)\circ \rho,&(\id_B\otimes \epsilon_V)\circ \rho&=\id_B. 
\end{align*} 
\item The unit of $B$ is a $V$-comodule morphism:
\[\rho(1_B)=1_B\otimes 1_V.\]
\item The product $m$ of $B$ is a $V$-comodule morphism:
\[\rho\circ m=(m\otimes  \cdot)\circ (\id \otimes c\otimes \id)\circ (\rho \otimes \rho),\]
where $c:B\otimes B\longrightarrow B\otimes B$ is the usual flip, sending $a\otimes b$ to $b\otimes a$.
\item The counit $\varepsilon_\Delta$ of $B$ is a $V$-comodule morphism:
\begin{align*}
&\forall x\in B,&(\varepsilon_\Delta\otimes \id)\circ \rho(x)&=\varepsilon_\Delta(x)1_V.
\end{align*}
\item The coproduct $\Delta$ of $B$ is a $V$-comodule morphism:
\[(\Delta\otimes \id)\circ \rho=m_{1,3,24}\circ (\rho\otimes \rho)\circ \Delta,\]
 where
\[m_{1,3,24}:\left\{\begin{array}{rcl}
B\otimes V\otimes B\otimes V&\longrightarrow&B\otimes B\otimes V\\
b_1\otimes v_2\otimes b_3\otimes v_4&\longrightarrow&b_1\otimes b_3\otimes v_2\cdot v_4. 
\end{array}\right.\]
\end{itemize}\end{defi}

Notice that the second and third items are equivalent to the fact that $\rho$ is an algebra morphism.

\begin{example}
\begin{itemize}
\item Let $(\Omega,\star)$ be a monoid and let $V=\K\Omega$ be the associated bialgebra. 
Let $B$ be a bialgebra over $V$. For any $\alpha\in \Omega$, we put
\[B_\alpha=\{x\in B\mid \rho(x)=x\otimes \alpha\}.\]
Then $\displaystyle B=\bigoplus_{\alpha \in \Omega} B_\alpha$. 
Indeed, if $x\in B$, we can write
\[\rho(x)=\sum_{\alpha\in \Omega} x_\alpha \otimes \alpha.\]
Then
\[(\rho\otimes \id)\circ \rho(x)=\sum_{\alpha\in \Omega} \rho(x_\alpha) \otimes \alpha
=(\id \otimes \delta_V)\circ \rho(x)=\sum_{\alpha\in \Omega} x_\alpha \otimes \alpha\otimes \alpha.\]
Therefore, for any $\alpha \in \Omega$, $\rho(x_\alpha)=x_\alpha \otimes \alpha$, that is to say $x_\alpha \in B_\alpha$.
Moreover,
\[x=(\id \otimes \epsilon_V)\circ \rho(x)=\sum_{\alpha\in \Omega} x_\alpha.\]

The second item of Definition \ref{defioverV} is equivalent to $1_B\in B_{1_\Omega}$. The third item is equivalent to
\begin{align*}
&\forall \alpha,\beta \in \Omega,&B_\alpha B_\beta&\subseteq B_{\alpha\star\beta}.
\end{align*}
The fourth item is equivalent to $\displaystyle \bigoplus_{\alpha \neq 1_\Omega}B_\alpha\subseteq \ker(\varepsilon_\Delta)$.
The last item is equivalent to
\begin{align*}
&\forall \alpha \in \Omega,&\Delta(B_\alpha)&\subseteq \bigoplus_{\alpha'\star \alpha''=\alpha}
B_{\alpha'}\otimes B_{\alpha''}.
\end{align*}
In other words, a bialgebra over $\K\Omega$ is an $\Omega$-graded bialgebra.
\item Let $V=\K(\Z/2\Z,\times)$.
A bialgebra over $V$ admits a decomposition 
$B=B_{\overline{0}}\oplus B_{\overline{1}}$, with $1_B\in B_{\overline{0}}$, $\varepsilon_\Delta(B_{\overline{1}})=(0)$, and
\begin{align*}
B_{\overline{0}} B_{\overline{1}}&+B_{\overline{1}} B_{\overline{0}}+B_{\overline{1}} B_{\overline{1}}\subseteq B_{\overline{1}},&B_{\overline{0}} B_{\overline{0}}&\subseteq B_{\overline{0}},\\
\Delta(B_{\overline{1}})&\subseteq B_{\overline{1}}\otimes B_{\overline{0}}+B_{\overline{0}}\otimes B_{\overline{1}}+B_{\overline{1}}\otimes B_{\overline{1}},&
\Delta(B_{\overline{0}})&\subseteq B_{\overline{0}}\otimes B_{\overline{0}}.
\end{align*}
In other words, a bialgebra over $V$ is a bialgebra with a decomposition $B=B_{\overline{0}}\oplus B_{\overline{1}}$,
such that $B_{\overline{0}}$ is a subbialgebra and $B_{\overline{1}}$ is a biideal. In particular, any bialgebra $(B,m,\Delta)$ is trivially a bialgebra over $V$,
with $B_{\overline{0}}=\K 1_B$ and $B_{\overline{1}}=\ker(\varepsilon_\Delta)$, or equivalently, for any $x\in B$,
\[\rho(x)= \varepsilon(x)1_B\otimes 1+(x-\varepsilon(x)1_B)\otimes X.\]
\item Let $\Omega$ be a finite monoid and let $\K[\Omega]$ be the bialgebra of functions over $G$, dual of the bialgebra $\K \Omega$.
A bialgebra over $\K[\Omega]$ is a family $(B,m,\Delta,\triangleleft)$ where $(B,m,\Delta)$ is a bialgebra and $\triangleleft$ is a right action of $\Omega$ on $B$ such that:
\begin{align*}
&\forall x,y\in B,&&\forall \omega \in \Omega,&
(xy)\triangleleft \omega&=(x\triangleleft \omega)
(y\triangleleft \omega),\\
&\forall x\in B,&&\forall \omega \in \Omega,&
\Delta(x\triangleleft \omega)&=\Delta(x)\triangleleft (\omega\otimes \omega),\\
&&&\forall \omega \in \Omega,&1_B\triangleleft \omega&=1_B,\\
&\forall x\in B,&&\forall \omega \in \Omega,&
\varepsilon_\Delta(x\triangleleft \omega)&=\varepsilon_\Delta(x).
\end{align*}

\end{itemize}
\end{example}

\begin{notation}
We shall use the Sweedler's notation $\rho(x)=x_0\otimes x_1$. The five items of Definition \ref{defioverV} become
\begin{align*}
(x_0)_0\otimes (x_0)_1\otimes x_1&=x_0\otimes x_1'\otimes x_1'',\\
x_0\varepsilon(x_1)&=x,\\
(1_B)_0\otimes (1_B)_1&=1_B\otimes 1_V,\\
(xy)_0\otimes (xy)_1&=x_0 y_0\otimes x_1 y_1,\\
\varepsilon_\Delta(x_0)x_1&=\varepsilon_\Delta(x)1_V,\\
(x_0)^{(1)}\otimes (x_0)^{(2)}\otimes x_1&=(x^{(1)})_0\otimes (x^{(2)})_0\otimes (x^{(1)})_1 (x^{(2)})_1.
\end{align*}
\end{notation}

\subsection{Antipode}

\begin{prop}\label{prop1.2}
Let $(V,m_V,\delta_V)$ be a bialgebra and let $(B,m,\Delta,\rho)$ be a bialgebra over $V$.
If $(B,m,\Delta)$ is a Hopf algebra of antipode $S$, then $S$ is a comodule morphism:
\[\rho\circ S=(S\otimes \id_V)\circ \rho.\]
\end{prop}

\begin{proof}
Let us give $\Hom(B,B\otimes V)$ its convolution product $*$: for any linear maps $f$, $g$ from $B$ to $B\otimes V$,
\[f*g=m_{B\otimes V}\circ (f\otimes g)\circ \Delta.\]
In this convolution algebra,
\begin{align*}
((S\otimes \id_V)\circ \rho) * \rho
&=m_{B\otimes V}\circ (S\otimes \id_V\otimes \id_B\otimes \id_V)\circ (\rho\otimes \rho)\circ \Delta\\
&=(m\circ (S\otimes \id_B)\circ \Delta \otimes \id_V)\circ m_{1,3,24}\circ (\rho\otimes \rho)\circ \Delta\\
&=(m\circ (S\otimes \id_B)\circ \Delta \otimes \id_V)\circ (\Delta \otimes \id)\circ \rho\\
&=(m\circ (S\otimes \id_B)\circ \Delta\otimes \id_V)\circ \rho\\
&=(\iota_B\circ \varepsilon_\Delta \otimes \id_V)\circ \rho\\
&=\iota_{B\otimes V}\circ \varepsilon_\Delta.
\end{align*}
So $(S\otimes \id_V)\circ \rho$ is a right inverse of $\rho$ in ($\Hom(B,B\otimes V),*)$.
\begin{align*}
\rho*(\rho\circ S)&=m_{B\otimes V}\circ (\rho \otimes \rho)\circ (\id \otimes S)\circ \Delta\\
&=\rho\circ m\circ (\id \otimes S)\circ \Delta\\
&=\rho \circ \iota_B\circ \varepsilon_\Delta\\
&=\iota_{B\otimes V}\circ \varepsilon_\Delta.
\end{align*}
So $\rho \circ S$ is a left inverse of  $\rho$ in ($\Hom(B,B\otimes V),*)$. As $*$ is associative,
 $(S\otimes \id_V)\circ \rho=\rho \circ S$.
\end{proof}

\begin{example}\begin{enumerate}
\item 
Let $(\Omega,\star)$ be a semigroup. If $V$ is the bialgebra of $(\Omega,\star)$,
we recover that if $B$ is an $\Omega$-graded bialgebra and a Hopf algebra, then, $S$ is $\Omega$-homogeneous, that is to say, for any $\alpha \in \Omega$,
\[S(B_\alpha)\subseteq B_\alpha.\]
\item Let $\Omega$ be a finite monoid.
If $(B,m,\Delta,\triangleleft)$ is a bialgebra over
$\K[\Omega]$ and a Hopf algebra, then 
for any $x\in B$, for any $\alpha \in \Omega$,
\[S(x\triangleleft \alpha)=S(x)\triangleleft \alpha.\]
\end{enumerate}
\end{example}

\subsection{Nonunitary cases}

We shall work with not necessarily  unitary bialgebras $(V,\cdot,\delta_V)$. If so, we put $uV=\K\oplus V$
and we give it a product and a coproduct defined as follows:
\begin{align*}
&\forall \lambda,\mu\in \K,&&\forall v,w\in V,&(\lambda+v)\cdot (\mu+w)&=\lambda\mu+\lambda w+\mu v+v\cdot w,\\
&\forall \lambda \in \K,&&\forall v\in V,&\delta_{uV}(\lambda+v)&=\lambda 1\otimes 1+\delta_V(v).
\end{align*}
Then $(uV,\cdot,\delta_{uV})$ is a counitary and unitary bialgebra, and $V$ is a nonunitary subbialgebra of $uV$. 

\begin{defi}\label{Defi1.7}
Let $(V,\cdot,\delta_V)$ be a not necessarily unitary bialgebra and $(uV,\cdot,\delta_{uV})$ be its unitary extension.
A bialgebra over $V$ is a bialgebra $(B,m,\Delta,\rho)$ over $uV$ such that
\begin{align*}
\rho(\ker(\varepsilon_\Delta))&\subseteq B\otimes V.
\end{align*}
\end{defi}

\begin{remark}
If $(B,m,\Delta,\rho)$  is a bialgebra over the nonunitary bialgebra $(V,\cdot,\delta_V)$, then
\[\{b\in B\mid \rho(b)=b\otimes 1\}=\K 1_B.\]
Indeed, if $\rho(b)=b\otimes 1$, putting $b'=b-\varepsilon_\Delta(b)1_B$, then $b'\in \ker(\varepsilon_\Delta)$. Hence,
\[\rho(b')=\rho(b)-\varepsilon_\Delta(b)1_B\otimes 1=(b-\varepsilon(b)1_B)\otimes 1\in B\otimes V,\]
so $b=\varepsilon_\Delta(b)1_B$.
\end{remark}

In the sequel, we will mention that we work with a nonunitary bialgebra $(V,\cdot,\delta_V)$ if we want to use Definition
\ref{Defi1.7} instead of Definition \ref{defioverV}, even if $(V,\cdot)$ has a unit -- that will happen when we will work with $\K$.

\begin{example}
\begin{enumerate}
\item If $\Omega$ is a semigroup, then  a bialgebra $(B,m,\Delta)$  over $\K\Omega$ is a connected $u\Omega$-graded bialgebra,
where $u\Omega=\{e\}\sqcup \Omega$ with the extension of the product of $\Omega$ such that $e$ is a unit:
\begin{align*}
&&B&=\bigoplus_{\alpha \in u\Omega}B_\alpha,\\
&\forall \alpha,\beta \in \Omega,&\Delta(B_\alpha)&\subseteq 
\sum_{\substack{\alpha',\alpha''\in \Omega,\\ \alpha'\times \alpha''=\alpha}}
B_{\alpha'}\otimes B_{\alpha''}+B_\alpha \otimes B_e+B_e\otimes B_\alpha,\\
&&B_e&=\K 1_B,\\
&\forall \alpha \in \Omega,&\varepsilon_\Delta(B_\alpha)&=(0).
\end{align*}
\item If $V=\K$\footnote{which is of course unitary, but which we treat as a nonunitary bialgebra, as mentioned before.}, 
as $u\K$ is isomorphic to $\K(\Z/2\Z,\times)$, 
any bialgebra $(B,m,\Delta)$ is a bialgebra over $V$
with $B_{\overline{0}}=\K 1_B$ and $B_{\overline{1}}=\ker(\varepsilon_\Delta)$. 
\end{enumerate}\end{example}

\subsection{Double bialgebras over $V$}

\begin{defi}\label{defi1.4}
Let $(B,m,\Delta,\delta)$ be a double bialgebra, $(V,\cdot,\delta_V)$ be a bialgebra and $\rho:B\longrightarrow B\otimes V$
be a right coaction of $V$ over $B$. We shall say that $(B,m,\Delta,\delta,\rho)$ is a double bialgebra over $V$
if $(B,m,\Delta,\rho)$ is a bialgebra over $V$ and
\[(\id \otimes c)\circ (\rho\otimes \id)\circ \delta=(\delta \otimes \id)\circ \rho:B\longrightarrow B\otimes B\otimes V,\]
where $c:V\otimes B\longrightarrow B\otimes V$ is the usual flip.
In other words, with Sweedler's notation $\delta(x)=x'\otimes x''$ for any $x\in B$,
\[(x')_0\otimes x''\otimes (x')_1=(x_0)'\otimes (x_0)''\otimes x_1.\]
\end{defi}

\begin{remark}
In other words, in a double bialgebra $B$ over $V$, considering the left coaction $\rho^{op}$ of $V^{cop}=(V,\delta_V^{op})$ on $B$, 
\[(\rho^{op}\otimes \id)\circ \delta=(\id \otimes \delta)\circ\rho^{op},\]
which means that $B$ is a $(V,\delta_V^{op})$--$(B,\delta)$-bicomodule.
\end{remark}

\begin{example}
Let $\Omega$ be a finite monoid. 
A double bialgebra $(B,m,\Delta,\triangleleft)$ over $\K[\Omega]$ is a bialgebra over $\K[\Omega]$ and a double bialgebra such that for any $x\in B$, for any $\alpha \in \Omega$,
\[\delta(x\triangleleft \alpha)=\delta(x)\triangleleft (\alpha \otimes e_\Omega),\]
where $e_\Omega$ is the unit of $\Omega$. 
\end{example}

In the nonunitary case:

\begin{defi}
Let $(V,\cdot,\delta_V)$ be a not necessarily unitary bialgebra.
A double bialgebra over $V$ is a double bialgebra $(B,m,\Delta,\delta,\rho)$ over $uV$ such that
$(B,m,\Delta,\rho)$ is a bialgebra over $V$.
\end{defi}

\begin{example}
\begin{enumerate}
\item Let $\Omega$ be a semigroup.
A double bialgebra $(B,m,\Delta,\delta)$ over $\K\Omega$ is a bialgebra over $\K\Omega$ such that for any $\alpha \in \Omega$,
\[\delta(B_\alpha)\subseteq B_\alpha \otimes B.\]
\item 
If $V=\K$, 
as $u\K$ is isomorphic to $\K(\Z/2\Z,\times)$, 
any double bialgebra $(B,m,\Delta,\delta)$ is a double bialgebra over $V$
with $B_{\overline{0}}=\K 1_B$ and $B_{\overline{1}}=\ker(\varepsilon_\Delta)$. 
\end{enumerate} \end{example}

\section{Quasishuffle bialgebras}

\subsection{Definition}

\cite{Ebrahimi-Fard2017-2,Foissy40,Hoffman2000,Hoffman2020} 
Let $(V,\cdot)$ be a nonunitary bialgebra. 
The tensor algebra $T(V)$ is given the quasishuffle product associated to $V$: For any $v_1,\ldots,v_{k+l}\in V$,
\[v_1\ldots v_k \squplus v_{k+1} \ldots v_{k+l}=\sum_{\sigma \in \QSh(k,l)} \left(\prod_{i\in\sigma^{-1}(1)}^\cdot  v_i\right)
\ldots \left(\prod_{i\in\sigma^{-1}(\max(\sigma))}^\cdot v_i\right),\]
where $\QSh(k,l)$ is the set of $(k,l$)-quasishuffles, that is to say surjections $\sigma:[k+l]\longrightarrow [\max(\sigma)]$
such that $\sigma(1)<\ldots <\sigma(k)$ and $\sigma(k+1)<\ldots <\sigma(k+l)$. 
The symbol $\displaystyle \prod^\cdot$ means that the corresponding products are taken in $(V,\cdot)$. 
The coproduct $\Delta$ is given by deconcatenation: for any $v_1,\ldots,v_n\in V$,
\[\Delta(v_1 \ldots  v_n)=\sum_{k=0}^n
v_1 \ldots  v_k\otimes v_{k+1} \ldots  v_n.\]
A special case is given when $\cdot$ is the zero product of $V$. In this case, we obtain the shuffle product $\shuffle$ of $T(V)$.
The bialgebra $(T(V),\shuffle,\Delta)$ is denoted by $\Sh(V)$. \\

If $(V,\cdot,\delta_V)$ is a not necessarily unitary commutative bialgebra, then $\qsh(V)$ inherits a second coproduct $\delta$
making it a double bialgebra. For any $v_1,\ldots,v_k\in V$, with Sweeder's notation $\delta_V(v)=v'\otimes v''$,
\[\delta(v_1 \ldots  v_n)=\sum_{1\leq i_1<\ldots<i_p<k}
\left(\prod_{1\leq i\leq i_1}^\cdot v'_i\right)\ldots  \left(\prod_{i_p+1\leq i\leq k}^\cdot v'_i\right)
\otimes (v''_1 \ldots v''_{i_1})\squplus \ldots \squplus (v''_{i_p+1} \ldots  v''_k).\]

\begin{prop}
Let $(V,\cdot,\delta_V)$ be a nonunitary bialgebra. We define a coaction of $V$ on $\qsh(V)$ by
\begin{align*}
&\forall v_1,\ldots,v_n \in V,&\rho(v_1\ldots v_n)&=v'_1\ldots v'_n \otimes v''_1\cdot \ldots \cdot v''_n.
\end{align*}
\begin{enumerate}
\item  The quasishuffle bialgebra $\qsh(V)=(T(V),\squplus,\Delta,\rho)$  is a bialgebra over $V$ if and only if
$(V,\cdot)$ is commutative. 
\item The quasishuffle double bialgebra $\qsh(V)=(T(V),\squplus,\Delta,\delta,\rho)$  is a bialgebra over $V$ if and only if
$(V,\cdot)$ is commutative and cocommutative.
\end{enumerate}
\end{prop}

\begin{proof}
1. Let us assume that $\qsh(V)$ is a double bialgebra over $V$ with this coaction $\rho$. For any $v,w\in V$,
\begin{align*}
\rho(v\squplus w)&=\rho(vw+wv+v\cdot w)\\
&=v'w'\otimes v''\cdot w''+w'v'\otimes w''\cdot v''+v'\cdot w'\otimes v''\cdot w'',\\
(\squplus \otimes m)\circ (\rho\otimes \rho)(v\otimes w)
&=v'\squplus w'\otimes v''\cdot w''\\
&=(v'w'+w'v'+v'\otimes w')\otimes v''\cdot w''.
\end{align*}
As $\squplus$ is comodule morphism, we obtain that for any $v,w\in V$, 
\[w'\otimes v'\otimes w''\cdot v''=w'\otimes v'\otimes v''\cdot w''.\]
Applying $\epsilon_V\otimes \epsilon_V\otimes \id_V$, this gives $v\cdot w=w\cdot v$, so $V$ is commutative.\\

Let us now assume that $V$ is commutative. 
The compatibilities of the unit and of the counit with the coaction $\rho$ are obvious. 
Let $v_1,\ldots,v_{k+l}\in V$ and let $\sigma\in \QSh(k,l)$. 
\begin{align*}
&\rho\left( \left(\prod_{i\in\sigma^{-1}(1)}^\cdot  v_i\right)\ldots  
\left(\prod_{i\in\sigma^{-1}(\max(\sigma))}^\cdot v_i\right)\right)\\
&= \left(\prod_{i\in\sigma^{-1}(1)}^\cdot  v_i\right)'\ldots  
\left(\prod_{i\in\sigma^{-1}(\max(\sigma))}^\cdot v_i\right)'
\otimes \left(\prod_{i\in\sigma^{-1}(1)}^\cdot  v_i\right)''\cdot \ldots \cdot 
\left(\prod_{i\in\sigma^{-1}(\max(\sigma))}^\cdot v_i\right)''\\
&= \left(\prod_{i\in\sigma^{-1}(1)}^\cdot  v_i'\right)\ldots  
\left(\prod_{i\in\sigma^{-1}(\max(\sigma))}^\cdot v_i'\right)
\otimes \left(\prod_{i\in\sigma^{-1}(1)}^\cdot  v_i''\right)\cdot \ldots \cdot 
\left(\prod_{i\in\sigma^{-1}(\max(\sigma))}^\cdot v_i''\right)\\
&= \left(\prod_{i\in\sigma^{-1}(1)}^\cdot  v_i'\right)\ldots  
\left(\prod_{i\in\sigma^{-1}(\max(\sigma))}^\cdot v_i'\right)
\otimes v''_1\cdot \ldots \cdot v''_n,
\end{align*}
as $(V,\cdot)$ is commutative. Summing over all possible $\sigma$, we obtain
\begin{align*}
\rho(v_1\ldots v_k\squplus v_{k+1}\ldots v_{k+l})
&=\left(\sum_{\sigma \in \QSh(k,l)} \left(\prod_{i\in\sigma^{-1}(1)}^\cdot  v_i'\right)\ldots  
\left(\prod_{i\in\sigma^{-1}(\max(\sigma))}^\cdot v_i'\right)\right)\otimes v''_1\cdot \ldots \cdot v''_n\\
&=(v'_1\ldots v'_k\squplus v'_{k+1}\ldots v'_{k+l})\otimes (v''_1\cdot \ldots \cdot v''_k) \cdot (v''_{k+1}\cdot \ldots \cdot v''_{k+l})\\
&=\rho(v_1\ldots v_k)\rho(v_{k+1}\ldots v_{k+l}).
\end{align*}
Let $v_1,\ldots,v_k \in V$. If $0\leq i\leq k$,
\begin{align*}
m_{1,3,24}\circ (\rho \otimes \rho)(v_1\ldots v_i \otimes v_{i+1}\ldots v_k)
&=v'_1\ldots v'_i \otimes v'_{i+1}\ldots v'_n \otimes v''_1\cdot \ldots \cdot v''_k.
\end{align*}
Summing over all possible $i$, we obtain
\begin{align*}
m_{1,3,24}\circ (\rho \otimes \rho)\circ \Delta(v_1\ldots v_k)
&=\left(\sum_{i=0}^k v'_1\ldots v'_i \otimes v'_{i+1}\ldots v'_k\right)\otimes v''_1\cdot \ldots \cdot v''_k\\
&=(\Delta \otimes \id)\circ \rho(v_1\ldots v_k).
\end{align*}

2. Let us assume that $\qsh(V)$ is a double bialgebra over $V$. By the first part of this proof, $V$ is commutative. For any $v\in V$,
\begin{align*}
(\id \otimes \delta_V)\circ \delta_V(v)&=(\delta_V \otimes \id)\circ \delta_V(v)\\
&=(\delta \otimes \id)\circ \rho(v)\\
&=(\id \otimes c)\circ (\rho \otimes \id)\circ \delta(v)\\
&=(\id \otimes c)\circ (\delta \otimes \id)\circ \delta(v)\\
&=(\id \otimes \delta_V^{op})\circ \delta_V(v).
\end{align*}
Applying $\epsilon_V\otimes \id \otimes \id$, we obtain that $\delta_V^{op}=\delta_V$, so $V$ is cocommutative.\\

Let us assume that $V$ is commutative and cocommutative. It is proved in \cite{Foissy40} that 
$\qsh(V)$ is a double bialgebra.  By the first item, $\qsh(V)$ is a bialgebra over $V$. For any $v_1,\ldots,v_n\in V$,
\begin{align*}
&(\delta \otimes \id)\circ \rho(v_1\ldots v_k)\\
&=\sum_{1\leq i_1<\ldots<i_p<k}
\left(\prod_{1\leq i\leq i_1}^\cdot v'_i\right)\ldots  \left(\prod_{i_p+1\leq i\leq k}^\cdot v'_i\right)
\otimes (v''_1 \ldots v''_{i_1})\squplus \ldots \squplus (v''_{i_p+1} \ldots  v''_k)
\otimes v'''_1\cdot \ldots \cdot v'''_k,
\end{align*}
whereas
\begin{align*}
&(\id \otimes c)\circ (\rho \otimes \id)\circ \delta(v_1\ldots v_k)\\
&=\sum_{1\leq i_1<\ldots<i_p<k}
\left(\prod_{1\leq i\leq i_1}^\cdot v'_i\right)'\ldots  \left(\prod_{i_p+1\leq i\leq k}^\cdot v'_i\right)'
\otimes (v''_1 \ldots v''_{i_1})\squplus \ldots \squplus (v''_{i_p+1} \ldots  v''_k)\\
&\otimes \left(\prod_{1\leq i\leq i_1}^\cdot v'_i\right)''\cdot\ldots\cdot \left(\prod_{i_p+1\leq i\leq k}^\cdot v'_i\right)''\\
&=\sum_{1\leq i_1<\ldots<i_p<k}
\left(\prod_{1\leq i\leq i_1}^\cdot v'_i\right)\ldots  \left(\prod_{i_p+1\leq i\leq k}^\cdot v'_i\right)
\otimes (v''_1 \ldots v''_{i_1})\squplus \ldots \squplus (v'''_{i_p+1} \ldots  v'''_k)\\
&\otimes \left(\prod_{1\leq i\leq i_1}^\cdot v''_i\right)\cdot\ldots\cdot \left(\prod_{i_p+1\leq i\leq k}^\cdot v''_i\right)\\
&=\sum_{1\leq i_1<\ldots<i_p<k}
\left(\prod_{1\leq i\leq i_1}^\cdot v'_i\right)\ldots  \left(\prod_{i_p+1\leq i\leq k}^\cdot v'_i\right)
\otimes (v'''_1 \ldots v'''_{i_1})\squplus \ldots \squplus (v'''_{i_p+1} \ldots  v'''_k)
\otimes v''_1\cdot \ldots \cdot v''_k,
\end{align*}
as $V$ is commutative. By the cocommutativity of $\delta_V$,
\[(\delta \otimes \id)\circ \rho(v_1\ldots v_k)=(\id \otimes c)\circ (\rho \otimes \id)\circ \delta(v_1\ldots v_k),\]
so $(T(V),\squplus, \Delta,\delta,\rho)$ is a double bialgebra over $V$.  
\end{proof}

\subsection{Universal property of quasishuffle bialgebras}

Let us recall the definition of connectivity for bialgebras:

\begin{notation}\begin{enumerate}
\item Let $(B,m,\Delta)$ be a bialgebra, of unit $1_B$ and of counit $\varepsilon_\Delta$. For any $x\in \ker(\varepsilon_\Delta)$,
we put
\[\tdelta(x)=\Delta(x)-x\otimes 1-1\otimes x.\] 
%This map is extended as a map from $B$ to $\ker(\varepsilon_\Delta)^{\otimes 2}$ by $\tdelta(1_B)=0$. 
Then $\tdelta$ is a coassociative coproduct on $\ker(\varepsilon_\Delta)$.
Its iterations will be denoted by $\tdelta^{(n)}:\ker(\varepsilon_\Delta)\longrightarrow
\ker(\varepsilon_\Delta)^{\otimes(n+1)}$, inductively defined by
\[\tdelta^{(n)}=\begin{cases}
\id_{\ker(\varepsilon_\Delta)}\mbox{ if }n=0,\\
(\tdelta^{(n-1)}\otimes \id)\circ \tdelta\mbox{ otherwise}.
\end{cases}\]
\item The bialgebra $(B,m,\Delta)$ is connected if
\begin{align*}
\ker(\varepsilon_\Delta)&=\bigcup_{n=0}^\infty \ker\left(\tdelta^{(n)}\right).
\end{align*}
\item If $(B,m,\Delta)$ is a connected bialgebra, we put, for $n\geq 0$,
\[B_{\leq n}=\K 1_B\oplus \ker\left(\tdelta^{(n)}\right).\]
As $B$ is a connected, this is a filtration of $B$, known as the coradical filtration \cite{Abe1980,Sweedler1969}.
 Moreover, for any $n\geqslant 1$, because of  the coassociativity of $\tdelta$,
\[\tdelta(B_{\leq n})\subseteq B_{\leq n-1}^{\otimes 2}.\]
\end{enumerate}\end{notation}

In the case of bialgebras over a bialgebra $(V,\cdot,\delta_V)$, the connectedness is sometimes automatic:

\begin{prop}\label{prop6.9}
Let $(V,\cdot,\Delta)$ be a nonunitary  bialgebra. For any $n\geqslant 1$, we put
\[V^{\cdot n}=\vect(v_1\cdot \ldots \cdot v_n,\: v_1,\ldots,v_n \in V).\]
If $\displaystyle \bigcap_{n\geqslant 1} V^{\cdot n}=(0)$, then any bialgebra over $V$ is a connected bialgebra.
\end{prop}

\begin{proof}
Let $(B,m,\Delta,\rho)$ be a bialgebra over $V$ and let $x\in \ker(\varepsilon_\Delta)$. 
We put
\[\rho(x)=\sum_{i=1}^p x_i\otimes v_i.\]
Let us denote by $W$ the vector space generated by the elements $v_i$. By definition, this is a finite-dimensional vector space
and $\rho(x)\in B\otimes W$. As $W$ is finite-dimensional, the decreasing sequence of vector spaces
$(W\cap V^{\cdot n})_{n\geqslant 1}$ is stationary, so there exists $N\geqslant 1$ such that if $n\geqslant N$,
$W\cap V^{\cdot n}=W\cap V^{\cdot N}$. Therefore
\[W\cap V^{\cdot N}=W\cap \bigcap_{n\geqslant 1}V^{\cdot n}=(0).\]
Moreover,
\begin{align*}
\underbrace{m_{1,3,\ldots,2N-1,24\ldots 2N}\circ \rho^{\otimes N}\circ \tdelta^{(N-1)}(x)}_{\in  B^{\otimes N}\otimes V^{\cdot N}}
&=\underbrace{(\tdelta^{(N-1)} \otimes \id)\circ \rho(x)}_{\in B^{\otimes N}\otimes W}.
\end{align*}
As $V^{\cdot N}\cap W=(0)$, $(\tdelta^{(N-1)} \otimes \id)\circ \rho(x)=0$.
Then
\begin{align*}
(\id^{\otimes N} \otimes \epsilon_V)\circ (\tdelta^{(N-1)} \otimes \id)\circ \rho(x)&=\tdelta^{(N-1)}(x)=0.
\end{align*}
So $(B,m,\Delta)$ is connected.
\end{proof}

\begin{example}
\begin{enumerate}
\item  If $(V,\cdot,\delta_V)$ is the bialgebra of the semigroup $(\N_{>0},+)$, 
then $\displaystyle \bigcap_{n\geqslant 1} V^{\cdot n}=(0)$.
We recover the classical result that any $\N$-graded bialgebra $B$ such that $B_0=\K 1_B$ is connected. 
This also works for algebras of semigroups $\N^n\setminus\{0\}$, for example. 
\item This does not hold if $V$ is unitary, as then $V^{\cdot n}=V$ for any $n\in \N$. 
\end{enumerate}\end{example}

\begin{theo} \label{theomorphismes}
Let $V$ be a nonunitary, commutative  bialgebra and let $(B,m,\Delta,\rho)$ be a connected bialgebra over $V$.
For any character $\lambda$ of $B$, there exists a unique morphism $\phi$ from $(B,m,\Delta,\rho)$ to 
$(T(V),\squplus,\Delta,\rho)$ of bialgebras over $V$ such that $\epsilon_\delta\circ \phi=\lambda$.
Moreover, for any $x\in \ker(\varepsilon_\Delta)$,
\begin{align}
\label{eqphisurV}
\phi(x)&=\sum_{n=1}^\infty \underbrace{((\lambda \otimes \id)\circ \rho)^{\otimes n}\circ \tdelta^{(n-1)}(x)}
_{\in V^{\otimes n}}.
\end{align}
\end{theo}

\begin{proof}
Let us first prove that for any $\lambda \in V^*$ such that $\lambda(1_B)=1$, there exists a unique coalgebra morphism
$\phi:(B,\Delta,\rho)\longrightarrow (T(V),\Delta,\rho)$ of coalgebras over $V$ such that $\epsilon_\delta\circ \phi=\lambda$. \\

\textit{Existence}. Let $\phi:B\longrightarrow \qsh(V)$ defined by (\ref{eqphisurV}) and by $\phi(1_B)=1$.
By connectivity of $B$,  (\ref{eqphisurV}) makes perfectly sense. Let us prove that $\phi$ is a coalgebra morphism. 
As $\phi(1_B)=1$, it is enough to prove that for any $x\in \ker(\varepsilon_\Delta)$,
$\tdelta\circ \phi(x)=(\phi\otimes \phi)\circ \tdelta(x)$.
We shall use Sweedler's notation $\tdelta^{(n-1)}(x)=x^{(1)}\otimes \ldots \otimes x^{(n)}$.
\begin{align*}
&\tdelta\circ \phi(x)\\
&=\sum_{n=1}^\infty\lambda\left(x^{(1)}_0\right)\ldots \lambda\left(x^{(n)}_0\right)
\tdelta\left(x^{(1)}_1\ldots x^{(n)}_1\right)\\
&=\sum_{n=1}^\infty \sum_{i=1}^{n-1}\lambda\left(x^{(1)}_0\right)\ldots \lambda\left(x^{(n)}_0\right)
x^{(1)}_1\ldots x^{(i)}_1\otimes x^{(i+1)}_1\ldots x^{(n)}_1\\
&=\sum_{i,j\geqslant 1} \lambda\left(x^{(1)(1)}_0\right)\ldots \lambda\left(x^{(1)(i)}_0\right)
\lambda\left(x^{(2)(1)}_0\right)\ldots \lambda\left(x^{(2)(j)}_0\right)x^{(1)(1)}_1\ldots x^{(1)(i)}_1
\otimes x^{(2)(1)}_1\ldots x^{(2)(j)}_1\\
&=(\phi\otimes \phi)\left(x^{(1)}\otimes x^{(2)}\right)\\
&=(\phi\otimes \phi)\circ \tdelta(x).
\end{align*}
Let us prove that $\epsilon_\delta\circ \phi=\lambda$. If $x=1_B$, then $\epsilon_\delta\circ \phi(1_B)=\epsilon_\delta(1)=1=\lambda(1_B)$.
If $x\in \ker(\varepsilon_\Delta)$, as $\epsilon_\delta(V^{\otimes n})=(0)$ for any $n\geq 2$,
\[\epsilon_\delta\circ \phi(x)=\epsilon_\delta\circ (\lambda \otimes \id)\circ \rho\circ \tdelta^{(0)}(x)+0
=\lambda\left(( \id \otimes \epsilon_\delta)\circ \rho (x)\right)=\lambda(x).\]
Let us prove that $\phi$ is a comodule morphism. If $x=1_B$, then 
\[\rho\circ \phi(1_B)=1\otimes 1=(\phi\otimes \id)(1_B\otimes 1)=(\phi\otimes \id)\circ \rho(1_B).\]
Let us assume that $x\in \ker(\varepsilon_\Delta)$. 
\begin{align*}
(\phi\otimes \id) \circ \rho(x)&=\phi(x_0)\otimes x_1\\
&=\sum_{n=1}^\infty \lambda \left((x_0)^{(1)}_0\right)\ldots 
\lambda \left((x_0)^{(n)}_0\right) (x_0)^{(1)}_1\ldots (x_0)^{(n)}_1 \otimes x_1\\
&=\sum_{n=1}^\infty\lambda\left(x^{(1)}_{00}\right)\ldots \lambda\left(x^{(n)}_{00}\right)
x^{(1)}_{01}\ldots x^{(n)}_{01}\otimes x^{(1)}_1\cdot \ldots x^{(n)}_1\\
&=\sum_{n=1}^\infty\lambda\left(x^{(1)}_0\right)\ldots \lambda\left(x^{(n)}_0\right)
x^{(1)}_1\ldots x^{(n)}_1\otimes x^{(1)}_2\cdot \ldots\cdot  x^{(n)}_2\\
&=\sum_{n=1}^\infty\lambda\left(x^{(1)}_0\right)\ldots \lambda\left(x^{(n)}_0\right) \rho\left(x^{(1)}_1\ldots x^{(n)}_1\right)\\
&=\rho\circ \phi(x).
\end{align*}

\textit{Uniqueness}. Let $\psi:(B,\Delta,\rho)\longrightarrow (T(V),\Delta,\rho)$ such that $\epsilon_\delta \circ \psi=\lambda$.
As $1$ is the unique group-like element of $\qsh(V)$, necessarily $\psi(1_B)=1=\phi(1_B)$.
It is now enough to prove that $\psi(x)=\phi(x)$ for any $x\in \ker(\varepsilon_\Delta)$.
We assume that $x\in B_{\leq n}$ and we proceed by induction on $n$. If $n=0$, there is nothing to prove.
Let us assume that $n\geq 1$. As $\tdelta(x)\in B_{\leq n-1}^{\otimes 2}$, by the induction hypothesis,
\[\tdelta \circ \psi(x)=(\psi \otimes \psi)\circ \tdelta(x)=(\phi \otimes \phi)\circ \tdelta(x)
=\tdelta \circ \phi(x),\]
so $\psi(x)-\phi(x)\in \ker(\tdelta)=V$. We put $\psi(x)-\phi(x)=v\in V$. Then
\begin{align*}
v&=(\epsilon_V\otimes \id)\circ \delta_V(v)\\
&=(\epsilon_\delta\otimes \id)\circ \rho(v)\\
&=(\epsilon_\delta\otimes \id) \circ \rho\circ \phi(x)-(\epsilon_\delta\otimes \id) \circ \rho\circ \psi(x)\\
&=(\epsilon_\delta\otimes \id) \circ (\phi\otimes \id)(x)-(\epsilon_\delta\otimes \id) \circ (\psi\otimes \id)(x)\\
&=(\lambda \otimes \id)(x)-(\lambda \otimes \id)(x)\\
&=0.
\end{align*}
So $\psi(x)=\phi(x)$.\\

Let us now consider a character $\lambda$. As $\lambda(1_B)=1$, we already proved that there exists a unique coalgebra morphism 
$\phi:(B,\Delta,\rho)\longrightarrow (T(V),\Delta,\rho)$ such that $\epsilon_\delta\circ \phi=\lambda$. 
Let us prove that it is an algebra morphism.
We consider the two morphisms $\phi_1=\squplus \circ (\phi\otimes \phi)$  and $\phi_2:\phi\circ m$, both
from $B\otimes B$ to $\qsh(V)$. As $\phi$, $\squplus$ and $m$ are both comodule and coalgebra morphisms,
$\phi_1$ and $\phi_2$ are comodule and coalgebra morphisms. Moreover, $B\otimes B$ is connected and,
as $\epsilon_\delta$ is a character of $(T(V),\squplus)$ and $\lambda$ is a character of $(B,m)$,
\begin{align*}
\epsilon_\delta\circ \squplus \circ (\phi\otimes \phi)&=(\epsilon_\delta\otimes \epsilon_\delta)\circ (\phi\otimes \phi)
=\lambda \otimes \lambda
=\lambda \otimes m
=\epsilon_\delta\circ \phi\circ m.
\end{align*}
So $\epsilon_\delta \circ \phi_1=\epsilon_\delta\circ \phi_2$. By the \emph{uniqueness} part, $\phi_1=\phi_2$.  \end{proof}

\begin{lemma}\label{lemmeKX}
\begin{enumerate}
\item The double bialgebras $\qsh(\K)=(T(\K),\squplus,\Delta,\delta)$ and $(\K[X],m,\Delta,\delta)$ are isomorphic,
through the map
\[\morH:\left\{\begin{array}{rcl}
\qsh(\K)&\longrightarrow&\K[X]\\
\lambda_1\ldots \lambda_n&\longrightarrow&\lambda_1\ldots \lambda_n H_n(X),
\end{array}\right.\]
where $H_n$ is the $n$-th Hilbert polynomial
\[H_n(X)=\frac{X(X-1)\ldots (X-n+1)}{n!}.\]
\item Let $V$ be a  nonunitary, commutative and cocommutative bialgebra.
The following map is a morphism of double bialgebras:
\[\morH_V:\left\{\begin{array}{rcl}
\qsh(V)&\longrightarrow&\K[X]\\
v_1\ldots v_n&\longrightarrow&\epsilon_V(v_1)\ldots \epsilon(v_n) H_n(X).
\end{array}\right.\]
\end{enumerate}
\end{lemma}

\begin{proof}
1. In order to simplify the reading of the proof, the element $1\in \K\subseteq \qsh(\K)$ is denoted by $x$.
We apply Theorem \ref{theomorphismes} with $B=\K[X]$, with its usual product $m$ and coproducts $\Delta$ and $\delta$. 
with the character $\epsilon_\delta$ of $\K[X]$, which sends any polynomial $P$ on $P(1)$. Let us denote by
$\phi$ the following morphism. Then $\phi(X)=\epsilon_\delta(X)x=x$. By multiplicativity, for any $n\geqslant 1$,
\[\phi(X^n)=x^{\squplus n}=n!x^n+\mbox{a linear span of $x^k$ with $k<n$}.\]
By triangularity, $\phi$ is an isomorphism. Let us denote by $\morH$ the inverse isomorphism, and 
let us prove that $\morH(x^n)=H_n(X)$ for any $n$ by induction on $n$. This obvious if $n=0$ or $1$. 
Let us assume that $n\geq 2$. Let us prove that for any $0\leq k\leq n-1$, $\morH(x^n)(k)=0$ by induction on $k$. 
As $\varepsilon_\Delta\circ \morH=\varepsilon_\Delta$, 
\[\morH(x^n)(0)=\varepsilon_\Delta \circ \morH(x^n)=\varepsilon_\Delta(x^n)=0.\]
If $k\geq 1$, as $\morH$ is a coalgebra morphism,
\begin{align*}
\morH(x^n)(k)&=\morH(x^n)(k-1+1)\\
&=\Delta\circ \morH(x^n)(k-1,1)\\
&=(\morH\otimes \morH)\circ \Delta(x^n)(k-1,k)\\
&=\sum_{l=0}^n \morH(x^l)(k-1)\morH(x^{n-l})(1)\\
&=\morH(x^n)(k-1)+\sum_{l=1}^{n-1}\morH_l(k-1)\morH_{n-l}(1)+\morH(x^n)(1)\\
&=\morH(x^n)(1),
\end{align*}
by the induction hypotheses on $k$ and $n$. 
As $\epsilon_\delta \circ \phi=\epsilon_\delta$, we obtain that $\epsilon_\delta \circ \morH=\epsilon_\delta$,
\[\morH(x^n)(1)=\epsilon_\delta\circ \morH(x^n)=\epsilon_\delta(x^n)=0.\]
Therefore, $\morH(x^n)$ is a multiple of $X(X-1)\ldots(X-n+1)$.  By triangularity of $\phi$, we obtain that 
\[\morH(x^n)=\frac{X^n}{n!}+\mbox{terms of degree $<n$}.\]
Consequently, $\morH(x^n)=H_n(X)$. \\

2. The counit $\epsilon_V:V\longrightarrow \K$ is a bialgebra morphism. By functoriality, we obtain 
a double bialgebra morphism from $\qsh(V)$ to $\qsh(\K)$,
which sends $v_1\ldots v_n \in V^{\otimes n}$ to 
$\epsilon_V(v_1)\ldots \epsilon_V(v_n) x^n$. Composing with the isomorphism of the preceding item, we obtain $\morH_V$.
\end{proof}

As any bialgebra is trivially a bialgebra over $\K$, we immediately obtain:

\begin{cor}
Let $(B,m,\Delta)$ be a connected bialgebra and let $\lambda$ be a character of $B$.
There exists a unique bialgebra morphism $\phi:(B,m,\Delta)\longrightarrow (\K[X],m,\Delta)$ such that 
for any $x\in B$, $\phi(x)(1)=\lambda(x)$. For any $x\in \ker(\varepsilon_\Delta)$,
\[\phi(x)=\sum_{n=1}^\infty \lambda^{\otimes n}\circ \tdelta^{(n-1)}(x) H_n(X).\]
\end{cor}

When $V$ is the bialgebra of the semigroup $(\N_{>0},+)$, we recover Aguiar, Bergeron and Sottile's result \cite{Aguiar2006-2},
with Proposition \ref{prop6.9}:

\begin{cor}
Let $(B,m,\Delta)$ be a graded bialgebra with $B_0=\K 1_B$ and let $\lambda$ be a character of $B$.
There exists a unique bialgebra morphism $\phi:(B,m,\Delta)\longrightarrow (\QSym,\squplus,\Delta)$ such that 
 $\epsilon_\delta\circ \phi=\lambda$. 
\end{cor}

\subsection{Double bialgebra morphisms}

\begin{theo} \label{theo7.10}
Let $V$ be a nonunitary, commutative and cocommutative bialgebra, 
and let $(B,m,\Delta,\delta,\rho)$ be a connected double bialgebra over $V$.
There exists a unique morphism $\phi$ from $(B,m,\Delta,\delta,\rho)$ to $(T(V),\squplus,\Delta,\delta,\rho)$ 
of double bialgebras over $V$. For any $x\in \ker(\varepsilon_\Delta)$,
\begin{align*}
\phi(x)&=\sum_{n=1}^\infty \underbrace{((\epsilon_\delta \otimes \id)\circ \rho)^{\otimes n}\circ \tdelta^{(n-1)}(x)}
_{\in V^{\otimes n}}.
\end{align*}
\end{theo}

\begin{proof}
\textit{Uniqueness}: such a morphism is a morphism $\phi$ from $(B,m,\Delta,\rho)$ to $(B,m,\Delta,\rho)$
with $\epsilon_\delta\circ \phi=\epsilon_\delta$. By Theorem \ref{theomorphismes}, it is unique.\\

\textit{Existence}: let $\phi:(B,m,\Delta,\rho)\longrightarrow (B,m,\Delta,\rho)$ be the (unique) morphism such that 
$\epsilon_\delta \circ \phi=\epsilon_\delta$. Let us prove that for any $x\in B_{\leq n}$,
$\delta \circ \phi(x)=(\phi\otimes \phi)\circ \delta(x)$ by induction on $n$. If $n=0$,
we can assume that $x=1_B$. Then
\[\delta \circ \phi(1_B)=(\phi\otimes \phi)\circ \delta(1_B)=1\otimes 1.\]
Let us assume the result at all ranks $<n$, with $n\geq 2$. Let $x\in \ker(\varepsilon_\Delta)$.
As $(\varepsilon_\Delta\otimes \id)\circ \delta(x)=\varepsilon_\Delta(x)1$, $\delta(x)\in \ker(\varepsilon_\Delta)\otimes B$. 
\begin{align*}
(\tdelta \otimes \id)\circ \delta\circ \phi(x)&=m_{1,3,24}\circ (\delta \otimes \delta)\circ \tdelta \circ \phi(x)\\
&=m_{1,3,24}\circ (\delta \otimes \delta)\circ (\phi\otimes \phi)\circ \tdelta(x)\\
&=m_{1,3,24}\circ (\phi\otimes \phi\otimes \phi\otimes \phi)\circ  (\delta \otimes \delta)\circ \tdelta(x)\\
&=(\phi\otimes \phi\otimes \phi)\circ m_{1,3,24}\circ (\delta \otimes \delta)\circ \tdelta(x)\\
&=(\phi\otimes \phi\otimes \phi)\circ (\tdelta \otimes \id)\circ \delta(x)\\
&=(\tdelta\otimes \id)\circ (\phi\otimes \phi)\circ \tdelta(x).
\end{align*}
We used the induction hypothesis on the both sides of the tensors appearing in $\tdelta(x)$ for the third equality.
We deduce that $(\delta \circ \phi-\phi\otimes \phi)\circ \delta(x) \in \ker(\tdelta\otimes \id)=V\otimes T(V)$.
Moreover,
\begin{align*}
(\id \otimes c)\circ (\rho \otimes \id)\circ \delta\circ \phi(x)
&=(\delta \otimes \id)\circ \rho\circ \phi(x)\\
&=(\delta \otimes \id)\circ (\phi\otimes \id)\circ \rho(x),\\
(\id \otimes c)\circ (\rho \otimes \id)\circ (\phi\otimes \phi)\circ \delta(x)
&=(\id \otimes c)\circ (\phi\otimes \id\otimes \phi)\circ (\rho \otimes \id)\circ \delta(x)\\
&=(\phi\otimes \phi\otimes \id)\circ (\id \otimes c)\circ (\rho\otimes \id)\circ \delta(x)\\
&=(\phi\otimes \phi\otimes \id)\circ (\delta\otimes \id)\circ \rho(x).
\end{align*} 
Putting $y=(\delta\circ \phi-\phi\otimes \phi)\circ \delta(x)\in V\otimes T(V)$, we proved that
\[(\id \otimes c)\circ (\rho\otimes \id)(y)=((\delta \circ \phi-(\phi\otimes \phi)\circ \delta)\otimes \id)\circ \rho(x).\]
As $y\in V\otimes T(V)$,
\[ \rho \otimes \id(y)=\delta_V\otimes \id(y).\]
Consequently,
\begin{align*}
(\epsilon_\delta\otimes \id \otimes \id)\circ (\rho \otimes \id)(y)
=(\epsilon_V \otimes \id \otimes \id)\circ (\delta_V\otimes \id)(y)=y.
\end{align*}
Moreover,
\begin{align*}
(\epsilon_\delta\otimes \id \otimes \id)\circ (\rho \otimes \id)(y)
&=(\epsilon_\delta\otimes \id \otimes \id)\circ (\delta \circ \phi\otimes \id)\circ \rho(x)\\
&-(\epsilon_\delta\otimes \id \otimes \id)\circ (((\phi\otimes \phi)\circ \delta)\otimes \id)\circ \rho(x)\\
&=(\phi\otimes \id)\circ \rho(x)-((((\epsilon_\delta \circ \phi)\otimes \phi)\circ \delta)\otimes \id)\circ \rho(x)\\
&=(\phi\otimes \id)\circ \rho(x)-(((\epsilon_\delta \otimes \phi)\circ \delta)\otimes \id)\circ \rho(x)\\
&=(\phi\otimes \id)\circ \rho(x)-(\phi\otimes \id)\circ \rho(x)\\
&=0.
\end{align*}
Hence, $y=0$, so $\delta\circ \phi(x)=(\phi\otimes \phi)\circ \delta(x)$. \end{proof}

Applying to $V=\K$ or $V=\K({>0},+)$:

\begin{cor} \label{cor7.8}
\begin{enumerate}
\item Let $(B,m,\Delta)$ be a connected double bialgebra.
There exists a unique double bialgebra morphism $\phi$ from $(B,m,\Delta,\delta)$  to $(\K[X],m,\Delta,\delta)$. 
For any $x\in \ker(\varepsilon_\Delta)$,
\[\phi(x)=\sum_{n=1}^\infty \epsilon_\delta^{\otimes n}\circ \tdelta^{(n-1)}(x) H_n(X).\]
\item Let $(B,m,\Delta)$ be a graded, connected double bialgebra, such that for any $n\in \N$,
\[\delta(B_n)\subseteq B_n \otimes B.\]
There exists a unique homogeneous double bialgebra morphism 
$\phi$ from $(B,m,\Delta,\delta)$ to $(\QSym,\squplus,\Delta,\delta)$. For any $x\in \ker(\varepsilon_\Delta)$,
\[\phi(x)=\sum_{n=1}^\infty \sum_{k_1,\ldots,k_n \geq 1}
\epsilon_\delta^{\otimes n}\circ (\pi_{k_1}\otimes \ldots \otimes \pi_{k_n})\circ \tdelta^{(n-1)}(x)
(k_1,\ldots,k_n).\]
\item Let $\Omega$ be a commutative monoid and let $(B,m,\Delta)$ be a connected $\Omega$-graded double bialgebra,
connected as a coalgebra, such that for any $\alpha \in \Omega$,
\[\delta(B_\alpha)\subseteq B_\alpha \otimes B.\]
There exists a unique homogeneous double bialgebra morphism 
$\phi$ from $(B,m,\Delta,\delta)$ to $\qsh(\K\Omega)$. For any $x\in \ker(\varepsilon_\Delta)$,
\[\phi(x)=\sum_{n=1}^\infty \sum_{\alpha_1,\ldots,\alpha_n \in \Omega}
\epsilon_\delta^{\otimes n}\circ (\pi_{\alpha_1}\otimes \ldots \otimes \pi_{\alpha_n})\circ \tdelta^{(n-1)}(x)
(\alpha_1,\ldots,\alpha_n).\]
\end{enumerate}
\end{cor}

As an application, let us give a generalization  of Hoffman's isomorphism between shuffle and quasishuffle algebras
\cite{Hoffman2000,Hoffman2020}:

\begin{theo}
Let $(V,\cdot)$ be a nonunitary, commutative algebra. The following map is a Hopf algebra isomorphism:
\begin{align*}
\Theta_V&:\left\{\begin{array}{rcl}
\Sh(V)=(T(V),\shuffle,\Delta)&\longrightarrow&\QSh(V)=(T(V),\squplus,\Delta)\\
w&\longrightarrow&\displaystyle \sum_{\substack{w=w_1\ldots w_k,\\ w_1,\ldots,w_k\neq \emptyset}}
\frac{1}{\ell(w_1)!\ldots \ell(w_k)!}|w_1|\ldots |w_k|,
\end{array}\right.
\end{align*}
where for any word $w$, $|w|$ is the product in $V$ of its letters, and $\ell(w)$ its length.
\end{theo}

\begin{proof}
We first prove this result when $(V,\cdot,\delta_V)$ is a commutative, cocommutative, counitary bialgebra, of counit $\epsilon_V$.
First, observe that $(T(V),\shuffle,\Delta,\rho)$ is a bialgebra over $(V,\cdot,\delta_V)$ and that the following map
is a character of $(T(V),\shuffle)$: for any word $w=x_1\ldots x_k$, 
\[\lambda(w)=\frac{1}{k!}\epsilon_V(x_1)\ldots \epsilon_V(x_k).\]
By the universal property of the quasishuffle algebra, there exists a unique Hopf algebra morphism
$\Theta_V:(T(V),\shuffle,\Delta)\longrightarrow (T(V),\squplus,\Delta)$ such that $\epsilon\circ \Theta_V=\lambda$.
For any word $w=v_1\ldots v_k$,
\begin{align*}
(\lambda \otimes \id)\circ \rho(w)&=\lambda(v'_1\ldots v'_k)\ v''_1\cdot \ldots \cdot v''_k\\
&=\frac{1}{k!}\epsilon_V(v'_1)\ldots \epsilon_V(v'_k)\ v''_1\cdot \ldots \cdot v''_k\\
&=\frac{1}{k!}v_1\cdot \ldots \cdot v_k\\
&=\frac{1}{\ell(w)!}|w|.
\end{align*}
Hence,
\begin{align*}
\Theta_V(w)&=\sum_{k=1}^\infty ((\lambda \otimes \id)\circ \rho)^{\otimes k}\circ \tdelta^{(k-1)}(w)\\
&=\sum_{\substack{w=w_1\ldots w_k,\\ w_1,\ldots,w_k\neq \emptyset}}
 ((\lambda \otimes \id)\circ \rho)^{\otimes k}(w_1\otimes \ldots \otimes w_k)\\
 &=\sum_{\substack{w=w_1\ldots w_k,\\ w_1,\ldots,w_k\neq \emptyset}}\frac{1}{\ell(w_1)!\ldots \ell(w_k)!}|w_1|\ldots |w_k|.
\end{align*}

Let us now consider an commutative algebra $(V,\cdot)$. Let $(S(V),m,\Delta)$ be the symmetric algebra
generated by $V$, with its usual product and coproduct.
Applying the first item to $S(V)$, we obtain a Hopf algebra morphism $\Theta_{S(V)}:(T(S(V)),\shuffle,\Delta)\longrightarrow
(T(S(V)),\squplus,\Delta)$. By restriction, we obtain a Hopf algebra morphism $\Theta_{S_+(V)}:
(T(S_+(V)),\shuffle,\Delta)\longrightarrow (T(S_+(V)),\squplus,\Delta)$. 
The canonical algebra morphism $\pi:S_+(V)\longrightarrow V$, sending $v_1\ldots v_k$ to $v_1\cdot \ldots \cdot v_k$
(which exists as $V$ is commutative), induces a surjective morphism $\varpi:T(S_+(V))\longrightarrow T(V)$,
which is obviously a Hopf algebra morphism from $(T(S_+(V)),\shuffle,\Delta)$ to $(T(V),\shuffle,\Delta)$ 
and from $(T(S_+(V)),\squplus,\Delta)$ to $(T(V),\squplus,\Delta)$. Moreover, the following diagram is commutative:
\[\xymatrix{(T(S_+(V)),\shuffle,\Delta)\ar[r]^{\Theta_{S^+(V)}}\ar[d]_\varpi &(T(S_+(V)),\squplus,\Delta)\ar[d]^\varpi \\
(T(V),\shuffle,\Delta)\ar[r]_{\Theta_V}&(T(V),\squplus,\Delta)}\]
As the vertical arrows are surjective Hopf algebra morphisms and the top horizontal arrow is also a Hopf algebra morphism,
the bottom horizontal arrow is also a Hopf algebra morphism.
For any word $w$, $\Theta_V(w)-w$ is a linear span of words of length $<\ell(w)$. 
By a triangularity argument, $\Theta_V$ is bijective.
\end{proof}

\begin{remark}
Using the same argument as in \cite{Hoffman2000}, it is not difficult to prove that for any nonempty word $w\in T(V)$,
\[\Theta_V^{-1}(w)=\sum_{\substack{w=w_1\ldots w_k,\\ w_1,\ldots,w_k\neq \emptyset}}
\frac{(-1)^{\ell(w)+k}}{\ell(w_1)\ldots \ell(w_k)}|w_1|\ldots |w_k|.\]
\end{remark}

It is immediate to show that $\Theta$ is a natural transformation from the functor $\Sh$ to the functor $\QSh$,
that is to say, if $\alpha:V\longrightarrow W$ is a morphism between two commutative non unitary algebras,
then $T(\alpha)\circ \Theta_V=\Theta_W\circ T(\beta)$, as Hopf algebra morphisms from $\Sh(V)$ to $\QSh(W)$.
Let us prove a unicity result:

\begin{prop}
Let $\Upsilon$ be a natural transformation from the functor $\Sh$ to the functor $\QSh$ (functors
from the category of commutative nonunitary algebras to the category of Hopf algebras).
There exists $\mu\in \K$ such that $\Upsilon=\Theta\circ \Phi^{(\mu)}$, where $\Phi^{(\mu)}$ is the natural transformation
from $\Sh$ to $\Sh$ defined for any commutative nonunitary algebra $V$ by
\begin{align*}
&\forall v_1,\ldots,v_n\in V,&\Phi^{(\mu)}_V(v_1\ldots v_n)&=\mu^n v_1\ldots v_n.
\end{align*}
\end{prop}

\begin{proof}
Let $\Upsilon$ be a natural transformation from $\Sh$ to $\QSh$. For any commutative nonunitary algebra $V$, 
let us denote by $\pi_V:T(V)\longrightarrow V$ the canonical projection on $V$ and let us put $\varpi_V=\pi_V\circ  \Upsilon_V$.
As $\Upsilon_V$ is an endomorphism of the cofree coalgebra $(T(V),\Delta)$, for any nonempty word $w\in T(V)$,
\begin{align}
\label{equpsilon}
\Upsilon_V(w)&= \sum_{\substack{w=w_1\ldots w_k,\\ w_1,\ldots,w_k\neq \emptyset}} \varpi_V(w_1)\ldots \varpi_V(w_k).
\end{align}

Let $V$ be the augmentation ideal of $\K[X_1,\ldots,X_n]$. We consider $(\lambda_1,\ldots,\lambda_n)\in \K^n$
and the endomorphism $\alpha$ of $V$ defined by $\alpha(x_i)=\lambda_i X_i$. By naturality of $\Upsilon$, 
\begin{align*}
\lambda_1\ldots \lambda_n \Upsilon_V(X_1\ldots X_n)&=
\Upsilon_V\circ T(\alpha)(X_1\ldots X_n)
=T(\alpha)\circ \Upsilon_V(X_1\ldots X_n).
\end{align*}
Applying $\pi_V$, we obtain
\[\lambda_1\ldots \lambda_n \varpi_V(X_1\ldots X_n)=\alpha \circ \varpi_V(X_1\ldots X_n).\]
Therefore, there exists $\mu_n\in \K$ such that
\[\varpi_V(X_1\ldots X_n)=\mu_n X_1\cdot \ldots \cdot X_n.\]
Let $W$ be any nonunitary commutative algebra, $v_1,\ldots,v_n\in V$  and let $\beta:V\longrightarrow W$ be the morphism defined
by $\phi(X_i)=v_i$. By naturality of $\Upsilon$,
\[T(\beta)\circ \Upsilon_V(X_1\ldots X_n)=\Upsilon_W\circ T(\beta)(X_1\ldots X_n).\]
Applying $\pi_W$, we obtain
\[\beta \circ \varpi_V(X_1\ldots X_n)=\beta(\mu_n X_1\cdot \ldots \cdot X_n)=\mu_n v_1\cdot \ldots \cdot v_n=\varpi_W(v_1\ldots v_n).\]
We proved the existence of a family of scalars $(\mu_n)_{n\geq 0}$ such that for any commutative nonunitary algebra $V$,
for any $v_1,\ldots,v_n\in V$, $\varpi_V(v_1\ldots v_n)=\mu_n v_1\cdot \ldots \cdot v_n$.\\

Let us study this sequence $(\mu_n)_{n\geq 0}$. Let $V$ be the augmentation ideal of $\K[X]$. 
For any $k,l\geq 1$, as $\Upsilon_V$ is an algebra morphism from $\Sh(V)$ to $\QSh(V)$,
\begin{align*}
\varpi_V(X^{\otimes k}\shuffle X^{\otimes l})&=\frac{(k+l)!}{k!l!}\varpi\left(X^{\otimes(k+l)}\right)\\
&=\frac{(k+l)!}{k!l!} \mu_{k+l}X^{k+l},\\
&=\pi_V\left(\Upsilon_V\left(X^{\otimes k}\shuffle X^{\otimes l}\right)\right)\\
&=\pi_V\left(\Upsilon_V\left(X^{\otimes k}\right)\squplus \Upsilon_V\left(X^{\otimes k}\right)\right)\\
&=\varpi_V\left(X^{\otimes k}\right)\cdot \varpi_V\left(X^{\otimes l}\right)\\
&=\mu_k\mu_l X^{k+l}.
\end{align*}
Hence, $(k+l)! \mu_{k+l}=k!\mu_k l!\mu_l$. This implies that for any $k\in \N$, $\mu_k=\dfrac{\mu^k}{k!}$, with $\mu=\mu_1$.
Therefore, by (\ref{equpsilon}), for any nonunitary commutative algebra, for any nonempty word $w\in T(V)$,
\begin{align*}
\Upsilon_V(w)&= \sum_{\substack{w=w_1\ldots w_k,\\ w_1,\ldots,w_k\neq \emptyset}} 
\frac{\mu^{\ell(w_1)+\ldots+\ell(w_k)}}{\ell(w_1)!\ldots \ell(w_k)!} |w_1|\ldots |w_k|\\
&= \mu^{\ell(w)}\sum_{\substack{w=w_1\ldots w_k,\\ w_1,\ldots,w_k\neq \emptyset}} 
\frac{1}{\ell(w_1)!\ldots \ell(w_k)!} |w_1|\ldots |w_k|\\
&=\Theta_V\circ \Phi^{(\mu)}_V(w). 
\end{align*}
In other words, $\Upsilon=\Theta\circ \Phi^{(\mu)}$. 
\end{proof}

\begin{remark}
For any $\mu\in \K$, for any commutative nonunitary algebra $V$ $\Phi^{(\mu)}_V$ is indeed a Hopf algebra
endomorphism of $\Sh(V)$, as $\Sh(V)$ is graded by the length of words. 
\end{remark}

\subsection{Action on bialgebra morphisms}

We here fix a bialgebra $(V,\cdot,\delta_V)$, nonunitary, commutative and cocommutative.

\begin{notation}
\begin{enumerate}
\item Let $(B,m,\Delta)$ and $(B',m',\Delta')$ be bialgebras. We denote by $M_{B\rightarrow B'}$ 
the set of bialgebra morphisms from
$(B,m,\Delta)$ to $(B',m',\Delta')$. 
\item Let $(B,m,\Delta,\rho)$ and $(B',m',\Delta',\rho')$ be bialgebras over $V$. We denote by $M^\rho_{B\rightarrow B'}$
the set of morphisms of bialgebra over $V$ from $B$ to $B'$, that is to say morphisms both of bialgebras 
and of comodules over $V$.
\end{enumerate}
\end{notation}

\begin{prop}\label{prop5.19}
Let $(B,m,\Delta,\delta,\rho)$ be a double bialgebra over $V$  and $(B',m',\Delta',\rho')$ be a bialgebra over $V$. 
The following map is a right action of the monoid of characters $(\Char(B),\star)$ attached to $(B,m,\delta)$ 
on $M^\rho_{B\rightarrow B'}$,
\begin{align*}
\leftsquigarrow&:\left\{\begin{array}{rcl}
M^\rho_{B\rightarrow B'}\times \Char(B)&\longrightarrow&M^\rho_{B\rightarrow B'}\\
(\phi,\lambda)&\longrightarrow&\phi\leftsquigarrow\lambda=(\phi\otimes \lambda)\circ \delta.
\end{array}\right.
\end{align*}
\end{prop}

\begin{proof}
Let $(\phi,\lambda)\in M^\rho_{B\rightarrow B'}\times \Char(B)$. Let us prove that $\psi=(\phi\otimes \lambda)\circ \delta$
is a bialgebra morphism. As $\phi$, $\lambda$ and $\delta$ are algebra morphisms, by composition
$\psi$ is an algebra morphism. 
\begin{align*}
\Delta'\circ \psi&=\Delta'\circ (\phi\otimes \lambda)\circ \delta\\
&=(\phi\otimes \phi)\circ \Delta \circ (\id \otimes \lambda)\circ \delta\\
&=(\phi\otimes \phi\otimes \lambda)\circ (\Delta \otimes \id)\circ \delta\\
&=(\phi\otimes \phi\otimes \lambda)\circ m_{1,3,24}\circ (\delta\otimes \delta)\circ \Delta\\
&=(\phi\otimes \lambda\otimes \phi\otimes \lambda)\circ (\delta\otimes \delta)\circ \Delta\\
&=(\psi \otimes \psi)\circ \Delta.
\end{align*}
We used that $\lambda$ is a character for the fifth equality. Moreover,
\begin{align*}
\varepsilon'_\Delta\circ \Psi&=(\varepsilon'_\Delta\otimes \lambda)\circ \delta
=\lambda \circ \eta \circ \varepsilon_\Delta
=\varepsilon_\Delta,
\end{align*}
as $\lambda(1_B)=1$ so $\lambda \circ \eta=\id_\K$. 
So $\psi \in M_{B\rightarrow B'}$. Let us now prove that $\psi$ is a comodule morphism.
As $\rho'\circ \phi=(\phi\otimes \id)\circ \rho$,
\begin{align*}
\rho' \circ \psi&=\rho'\circ(\phi\otimes \lambda)\circ \delta\\
&=(\phi\otimes \id \otimes \lambda)\circ (\rho\otimes \id)\circ \delta\\
&=(\phi\otimes \id \otimes \lambda)\circ (\id \otimes c)\circ (\delta\otimes \id)\circ \rho\\
&=(\phi\otimes \lambda\otimes \id)\circ (\delta\otimes \id)\circ \rho\\
&=(\psi\otimes \id)\circ \rho.
\end{align*}
So $\psi\in M^\rho_{B\rightarrow B'}$.\\

Let $\phi\in M^\rho_{B\rightarrow B'}$, $\lambda,\mu\in \Char(B)$.
\begin{align*}
(\phi\leftsquigarrow\lambda)\leftsquigarrow\mu&=(\phi\otimes \lambda \otimes \mu)\circ (\delta \otimes \id)\circ \delta\\
&=(\phi\otimes \lambda \otimes \mu)\circ (\id \otimes \delta)\circ \delta\\
&=(\phi \otimes \lambda \star \mu)\circ \delta\\
&=\phi \leftsquigarrow(\lambda \star \mu).
\end{align*}
Moreover,
\[\phi \leftsquigarrow \epsilon_\delta=(\phi \otimes \epsilon_\delta)\circ \delta=\phi.\]
Therefore,  $\leftsquigarrow$ is an action. \end{proof}

Moreover, any bialgebra morphism is compatible with these actions:

\begin{prop}\label{propactionsmorphismes}
Let $(B,m,\Delta,\delta,\rho)$ be a double bialgebra over $V$  and $B'$ and $B''$ be bialgebras over $V$. 
For any morphisms $\phi:B\longrightarrow B'$ and $\psi:B'\longrightarrow B''$ of bialgebras over $V$,
for any character $\lambda$ of $B$,
\[(\psi\circ \phi)\leftsquigarrow \lambda=\psi\circ (\phi\leftsquigarrow \lambda).\]
\end{prop}

\begin{proof}
Indeed,
\begin{align*}
(\psi\circ \phi)\leftsquigarrow \lambda&=((\psi\circ \phi)\otimes \lambda)\circ \delta
=\psi\circ (\phi\otimes \lambda)\circ \delta=\psi\circ (\phi\leftsquigarrow \lambda). \qedhere
\end{align*}
\end{proof}

\begin{cor}\label{cormorphismedouble}
Let $(B,m,\Delta,\delta,\rho)$ be a connected double bialgebra over $V$.
Let us denote by $\phi_1:B\longrightarrow \qsh(V)$ the unique morphism of double bialgebras of Theorem \ref{theo7.10}.
The following maps are bijections, inverse one from the other:
\begin{align*}
\theta&:\left\{\begin{array}{rcl}
\Char(B)&\longrightarrow&M^\rho_{B\rightarrow \qsh(V)}\\
\lambda&\longrightarrow&\phi_1\leftsquigarrow \lambda,
\end{array}\right.&
\theta'&:\left\{\begin{array}{rcl}
M^\rho_{B\rightarrow \qsh(V)}&\longrightarrow&\Char(B)\\
\phi&\longrightarrow&\epsilon_\delta \circ \phi.
\end{array}\right.&
\end{align*}
\end{cor}

\begin{proof}
Let $\phi \in M^\rho_{B\rightarrow \qsh(V)}$. We put $\phi'=\theta \circ \theta'$ and $\lambda=\epsilon_\delta \circ \phi$.
Then
\[\epsilon_\delta \circ \phi'=\epsilon_\delta \circ (\phi_1\leftsquigarrow \lambda)
=(\epsilon_\delta \circ \phi_1)\star \lambda=\epsilon_\delta \star \lambda=\lambda=\epsilon_\delta \circ \phi.\]
By the uniqueness in Theorem \ref{theomorphismes}, $\phi=\phi'$.\\

Let $\lambda \in \Char(B)$ and let $\lambda'=\theta'\circ \theta(\lambda)$. Then
\[\lambda'=\epsilon_\delta \circ (\phi_1\leftsquigarrow \lambda)=(\epsilon_\delta\circ \phi_1\otimes \lambda)\circ \delta
=(\epsilon_\delta \otimes \lambda)\circ \delta=\epsilon_\delta\star \lambda=\lambda.\]
So $\theta$ and $\theta'$ are bijections, inverse one from the other. \end{proof}

\begin{cor}\label{cor6.15}
\begin{enumerate}
\item Let $(B,m,\Delta,\delta)$ be a connected double bialgebra.
Let us denote by $\phi_1$ the unique morphism of double bialgebras from $B$ to $\K[X]$ of Theorem \ref{theo7.10}.
The following maps are bijections, inverse one from the other:
\begin{align*}
\theta&:\left\{\begin{array}{rcl}
\Char(B)&\longrightarrow&M_{B\rightarrow \K[X]}\\
\lambda&\longrightarrow&\phi_1\leftsquigarrow \lambda,
\end{array}\right.&
\theta'&:\left\{\begin{array}{rcl}
M_{B\rightarrow \K[X]}&\longrightarrow&\Char(B)\\
\phi&\longrightarrow&\epsilon_\delta \circ \phi.
\end{array}\right.&
\end{align*}
\item Let $(B,m,\Delta,\delta)$ be a connected, graded double bialgebra such that for any $n\in \N$,
\[\delta(B_n)\subseteq B_n \otimes B.\]
Let us denote by $\phi_1$ the unique homogeneous morphism of double bialgebras from $B$ to $\QSym$ of Theorem \ref{theo7.10}.
We denote by $M^0_{B\rightarrow \QSym}$ the set of bialgebra morphisms from $(B,m,\Delta)$
to $(\QSym, \squplus,\Delta)$ which are homogeneous of degree $0$.
The following maps are bijections, inverse one from the other:
\begin{align*}
\theta&:\left\{\begin{array}{rcl}
\Char(B)&\longrightarrow&M^0_{B\rightarrow \QSym}\\
\lambda&\longrightarrow&\phi_1\leftsquigarrow \lambda,
\end{array}\right.&
\theta'&:\left\{\begin{array}{rcl}
M^0_{B\rightarrow \QSym}&\longrightarrow&\Char(B)\\
\phi&\longrightarrow&\epsilon_\delta \circ \phi.
\end{array}\right.&
\end{align*}
\item Let $\Omega$ be a commutative monoid and  let $(B,m,\Delta,\delta)$ be a connected, $\Omega$-graded double bialgebra,
connected as a coalgebra, such that for any $\alpha\in \Omega$,
\[\delta(B_\alpha)\subseteq B_\alpha \otimes B.\]
Let us denote by $\phi_1$ the unique homogeneous morphism of double bialgebras 
from $B$ to $\qsh(\K\Omega)$ of Theorem \ref{theo7.10}.
We denote by $M^0_{B\rightarrow \qsh(\K\Omega)}$ the set of bialgebra morphisms from $(B,m,\Delta)$
to $\qsh(\K\Omega)$ which are homogeneous of degree the unit of $\Omega$.
The following maps are bijections, inverse one from the other:
\begin{align*}
\theta&:\left\{\begin{array}{rcl}
\Char(B)&\longrightarrow&M^0_{B\rightarrow \qsh(\K\Omega)}\\
\lambda&\longrightarrow&\phi_1\leftsquigarrow \lambda,
\end{array}\right.&
\theta'&:\left\{\begin{array}{rcl}
M^0_{B\rightarrow \qsh(\K\Omega)}&\longrightarrow&\Char(B)\\
\phi&\longrightarrow&\epsilon_\delta \circ \phi.
\end{array}\right.&
\end{align*}
\end{enumerate}
\end{cor}

\subsection{Applications to graphs}

We postpone the detailed construction of the double bialgebras of $V$-decorated graphs to a forthcoming paper \cite{Foissy41}. 
For any nonunitary commutative bialgebra $(V,\cdot,\delta_V)$, we obtain a double bialgebra over $V$ of 
$V$-decorated graphs $\calH_V[\bfG]$, generated by graphs $G$ which any vertex $v$ is decorated by an element $d_G(v)$, 
with conditions of linearity in each vertex.
For example, if $v_1,v_2,v_3,v_4\in V$ and $\lambda_2,\lambda_4\in \K$,
if $w_1=v_1+\lambda_2v_2$ and $w_2=v_3+\lambda_4 v_4$,
\begin{align*}
\tddeux{$w_1$}{$w_2$}\:&=\tddeux{$v_1$}{$v_3$}\:+\lambda_4\tddeux{$v_1$}{$v_4$}\:
+\lambda_2\tddeux{$v_2$}{$v_3$}\:+\lambda_2\lambda_4\tddeux{$v_2$}{$v_4$}\:.
\end{align*}
The product is given by the disjoint union of graphs, the decorations being untouched. For any graph $G$,
for any $X\subseteq V(G)$, we denote by $G_{\mid X}$ the graph defined by
\begin{align*}
G_{\mid X}&=X,&E(G_{\mid X})&=\{\{x,y\}\in E(G)\mid x,y\in X\}. 
\end{align*}
Then 
\[\Delta(G)=\sum_{V(G)=A\sqcup B} G_{\mid A}\otimes G_{\mid B},\]
the decorations being untouched. For any equivalence relation $\sim$ on $V(G)$:
\begin{itemize}
\item $G/\sim$ is the graph defined by
\begin{align*}
V(G/\sim)&=V(G)/\sim,&
E(G/\sim)&=\{\{\overline{x},\overline{y}\}\mid \{x,y\}\in E(G),\: \overline{x}\neq \overline{y}\},
\end{align*}
where for any $z\in V(G)$, $\overline{z}$ is its class in $V(G)/\sim$. 
\item $G\mid  \sim$ is the graph defined by
\begin{align*}
V(G\mid \sim)&=V(G),&
E(G\mid \sim)&=\{\{x,y\}\in E(G)\mid x\sim y\}.
\end{align*}
\item We shall say that $\sim \in \eq_c[G]$ if for any equivalence class $X$ of $\sim$, $G_{\mid X}$ is connected.
\end{itemize}
With these notations, the second coproduct $\delta$ is given by
\[\delta(G)=\sum_{\sim \in \eq_c[G]} G/\sim \otimes G\mid \sim.\]
Any vertex $w\in V(G/\sim)=V(G)/\sim$ is decorated by
\[\prod_{v\in w}^\cdot d_G(v)',\]
where the symbol $\displaystyle \prod^\cdot$ means that the product is taken in $V$ (recall that any vertex of $V(G/\sim)$
is a subset of $V(G)$). Any vertex $v\in V(G\mid \sim)=V(G)$ is decorated by $d_G(v)''$.
We use Sweedler's notation $\delta_V(v)=v'\otimes v''$, and it is implicit that in the expression of $\delta(G)$, 
everything is developed by multilinearity in the vertices. For example, if $v_1,v_2,v_3\in V$,
\begin{align*}
\Delta(\tdtroisdeux{$v_1$}{$v_2$}{$v_3$}\:)&=\tdtroisdeux{$v_1$}{$v_2$}{$v_3$}\:\otimes 1
+1\otimes \tdtroisdeux{$v_1$}{$v_2$}{$v_3$}\:+\tddeux{$v_1$}{$v_2$}\:\otimes \tdun{$v_3$}\:
+\tddeux{$v_2$}{$v_3$}\:\otimes \tdun{$v_1$}\:+\tdun{$v_1$}\:\tdun{$v_3$}\:\otimes \tdun{$v_2$}\:\\
&+\tdun{$v_3$}\:\otimes  \tddeux{$v_1$}{$v_2$}\:+\tdun{$v_1$}\:\otimes  \tddeux{$v_2$}{$v_3$}\:
+\tdun{$v_2$}\:\otimes \tdun{$v_1$}\:\tdun{$v_3$}\:,\\
\\
\delta(\tdtroisdeux{$v_1$}{$v_2$}{$v_3$}\:)&=
\tdtroisdeux{$v_1'$}{$v_2'$}{$v_3'$}\:\otimes \tdun{$v_1''$}\hspace{2mm}\tdun{$v_2''$}\hspace{2mm}
\tdun{$v_3''$}\hspace{2mm}+\tdun{$v_1'\cdot v_2'\cdot v_3'$}\hspace{10mm}\otimes 
\tdtroisdeux{$v_1''$}{$v_2''$}{$v_3''$}\:
+\tddeux{$v_1'\cdot v_2'$}{$v_3'$}\hspace{6mm} \otimes \tddeux{$v_1''$}{$v_2''$}\hspace{2mm}\tdun{$v_3''$}\hspace{2mm}
+\tddeux{$v_1'$}{$v_2'\cdot v_3'$}\hspace{6mm} \otimes \tddeux{$v_2''$}{$v_3''$}\hspace{2mm}\tdun{$v_1''$}\hspace{2mm}.
\end{align*}
For any $V$-decorated graph,
\[\epsilon_\delta(G)=\begin{cases}
\displaystyle \prod_{v\in V(G)} \epsilon_V(d_G(v)) \mbox{ if }E(G)=\emptyset,\\
0\mbox{ otherwise}. 
\end{cases}\]

\begin{prop}
For any graph $G$, we denote by $\col(G)$ the set of packed valid colourations of $G$, that is to say
surjective maps $c:V[G]\longrightarrow [\max(f)]$ such that for any $\{x,y\}\in E(g)$, $c(x)\neq c(y)$. 
We denote by $\Phi_1$ the unique morphism of double bialgebras over $V$ from $\calH_V[\bfG]$ to $\qsh(V)$. 
For any $V$-decorated graph $G$,
\[\Phi_1(G)=\sum_{c\in \col(G)} \left(\prod_{c(x)=1)}^\cdot d_V(x),\ldots,
\prod_{c(x)=\max(c))}^\cdot d_V(x)\right),\]
where for any vertex $x\in V(G)$, $d_V(x)\in V$ is its decoration.
\end{prop}

\begin{proof}
Let $G$ be a $V$-decorated graph. For any vertex $i$ of $G$, we denote by $v_i\in V$ the decoration of $i$.
The number of vertices of $G$ is denoted by $n$.
\begin{align*}
\Phi_1(G)&=\sum_{k=1}^n\sum_{\substack{V(G)=I_1\sqcup \ldots \sqcup I_k,\\ I_1,\ldots,I_k\neq \emptyset}}
\epsilon_\delta(G_{\mid I_1})\ldots \epsilon_\delta(G_{\mid I_k})
\left(\prod_{i\in I_1}^\cdot v_i,\ldots,\prod_{i\in I_k}^\cdot v_i\right)\\
&=\sum_{k=1}^n\sum_{c:V[G]\longrightarrow [k]\mbox{\scriptsize, surjective}}
\epsilon_\delta(G_{\mid c^{-1}(1)})\ldots \epsilon_\delta(G_{\mid c^{-1}(k)})
\left(\prod_{c(x)=1}^\cdot d_V(x),\ldots,\prod_{c(x)=k}^\cdot d_V(x) \right)\\
&=\sum_{c\in \col(G)} \left(\prod_{c(x)=1)}^\cdot d_V(x),\ldots,
\prod_{c(x)=\max(c))}^\cdot d_V(x)\right),
\end{align*}
as for any surjective map $c:V[G]\longrightarrow [\max(f)]$, 
\[\epsilon_\delta(G_{\mid c^{-1}(1)})\ldots \epsilon_\delta(G_{\mid c^{-1}(k)})=\begin{cases}
1\mbox{ if }c\in \col(G),\\
0\mbox{ otherwise}. 
\end{cases}\qedhere\]
\end{proof}

\begin{example} For any $v_1,v_2,v_3\in V$,
\begin{align*}
\Phi_1(\tddeux{$v_1$}{$v_2$}\:)&=v_1v_2+v_2v_1,\\
\Phi_1(\tdtroisdeux{$v_1$}{$v_2$}{$v_3$}\:)&=v_1v_2v_3+v_1v_3v_2+v_2v_1v_3+v_2v_3v_1+v_3v_1v_2+v_3v_2v_1
+(v_1\cdot v_3)v_2+v_2(v_1\cdot v_3),\\
\Phi_1(\hspace{1mm}\gdtroisun{\hspace{-4mm}$v_1$}{$v_3$}{\hspace{-1mm}$v_2$}\:)
&=v_1v_2v_3+v_1v_3v_2+v_2v_1v_3+v_2v_3v_1+v_3v_1v_2+v_3v_2v_1.
\end{align*}
\end{example}

If $V=\K$, we obtain the double bialgebra morphism $\phi_{chr}:\calH[\bfG]\longrightarrow \K[X]$, 
sending any graph on its chromatic polynomial. If $V$ is the algebra of the semigroup $({>0},+)$,
we obtain the morphism $\Phi_{chr}:\calH_V[\bfG]\longrightarrow \QSym$, sending any graphs
which vertices are decorated by positive integers to its chromatic (quasi)symmetric function \cite{Rosas2001}.

\bibliographystyle{amsplain}
\bibliography{biblio}

\end{document}